\documentclass[reqno,11pt]{amsart} 

\usepackage[top=25truemm,bottom=25truemm,left=25truemm,right=25truemm]{geometry}
\usepackage[bookmarksnumbered,plainpages]{hyperref}
\usepackage{amsmath,amsthm,amssymb,amsfonts,amscd,mathtools,color}
\usepackage{luatex85} 
\usepackage{graphicx,xcolor} 
\usepackage{indentfirst}
\usepackage{here}
\usepackage{inputenc,enumerate,calligra,latexsym}
\usepackage{mathrsfs,bm,dsfont,bbm,extpfeil,textcomp,ascmac,appendix}
\usepackage{tikz,tikz-3dplot,tikz-cd}
\numberwithin{equation}{section}

\theoremstyle{plain}
	\newtheorem{thm}{Theorem}[section]
	\newtheorem*{thm*}{Theorem}
	\newtheorem{cor}[thm]{Corollary}
	\newtheorem{lem}[thm]{Lemma}
	
	\newtheorem{pro}[thm]{Proposition}
	
	\newtheorem*{conj*}{Conjecture}
	
\theoremstyle{definition}
	\newtheorem{dfn}[thm]{Definition}
	\newtheorem{exmp}[thm]{Example}

\theoremstyle{remark}
  \newtheorem{cl}{Claim}

	\newtheorem*{pf}{Proof}

\def\Hom{{\rm Hom}}

\def\Hom{{\rm Hom}}

\def\Ext{{\rm Ext}}

\def\HMF{{\rm HMF}}
\def\Jac{{\rm Jac}}
\def\CM{{\rm CM}}
\def\Cker{{\rm Coker}}
\def\MF{{\rm MF}}

\def\rk{{\rm rk}}

\def\Ob{{\rm Ob}}
\def\syz{{\rm syz}}
\def\gr{{\rm gr}}

\def\C{\mathbb{C}}
\def\Z{\mathbb{Z}}
\newcommand{\mf}[4]{
\begin{tikzcd}[ampersand replacement=\&]
  {#1} \& {#4}
  \arrow["{#2}", shift left, from=1-1, to=1-2]
  \arrow["{#3}", shift left, from=1-2, to=1-1]
\end{tikzcd}
}
\newcommand{\pathTriangle}[4][]{
\path[#1] (#2)--(#3)--(#4)--cycle}
\begin{document}
\title[A category of graded matrix factorizations of a deformed polynomial]%
{A category of graded matrix factorizations of
a deformed polynomial associated to the $A_{\mu}\textrm{-}$singularity}

\author[T. NAKATANI]{Tomoya NAKATANI}
\address{Department of Mathematics and Informatics, 
Graduate School of Science and Engineering, Chiba University,
Yayoicho 1-33, Inage, Chiba, 263-8522 Japan.}
\email{tnakatani@g.math.s.chiba-u.ac.jp}

\begin{abstract}
  We discuss a triangulated category of graded matrix factorizations of 
  a deformed polynomial associated to the $A_{\mu}\textrm{-}$singularity.
  The semi-universal deformation of the $A_{\mu}\textrm{-}$singularity
  is given by a certain deformation of the polynomial of type $A_{\mu}$.
  In this paper, we consider the category of graded matrix factorizations
  associated to this deformed polynomial for a fixed parameter.
  To do so, we introduce a formal variable to make the polynomial homogeneous.
  As our main result, we construct a full strongly exceptional collection
  in this category for a generic parameter. 
\end{abstract}

\maketitle
\thispagestyle{empty}

\setcounter{tocdepth}{1}
\tableofcontents
\section{Introduction}

In his 1994 ICM address \cite{10.1007/978-3-0348-9078-6_11},
Kontsevich proposed \textit{Homological Mirror Symmetry (HMS) conjecture}
as a categorical formulation of mirror symmetry for Calabi-Yau manifolds.
In more detail, the HMS conjecture is the equivalence between the derived category
constructed from the Fukaya category of a Calabi-Yau manifold and 
the derived category of the coherent sheaves on a mirror Calabi-Yau manifold.
The HMS is currently being studied in a broader context beyond Calabi-Yau manifolds,
including Landau-Ginzburg models.

The triangulated categories of (ungraded) matrix factorizations were first
introduced and developed by Eisenbud \cite{eisenbud1980homological} and 
Kn\"{o}rrer \cite{Knorrerperiod} in the study of the maximal
Cohen-Macaulay modules over hypersurfaces. On the other hand,
the categories of matrix factorizations appear in string theory as the categories of 
topological D-branes of type B in Landau-Ginzburg models 
(see \cite{AntonKapustin_2003},\cite{atmp/1112627039}).
In \cite{takahashi2005matrix}, motivated by the work on the categories of 
topological D-branes of type B in Landau-Ginzburg orbifolds by Hori-Walcher \cite{hori2005f},
A. Takahashi introduced a triangulated category
of \textit{graded} matrix factorizations, where the orbifolding corresponds to the grading.
Independently, Orlov defines a triangulated category, called the category of graded 
D-branes of type B, which is in fact equivalent to the one Takahashi introduced.
In \cite{orlov2009derived}, Orlov also established the equivalence between triangulated categories
of graded singularities of hypersurfaces and triangulated categories of graded matrix factorizations.
In relation to their work, Kajiura-Saito-Takahashi introduced 
the category of graded matrix factorizations $\HMF^{gr}_{S}(f)$ of a weighted homogeneous polynomial $f \in S$
by using the Orlov's construction with a slight modification of the scaling of degrees in \cite{kajiura2007matrix}. 
In their paper, for a polynomial $f$ of type $ADE$, they showed that $\HMF^{gr}_{S}(f)$ is equivalent to 
the derived category of modules over the path algebra of the Dynkin quiver of the same type as $f$.

For example, we recall the case that a polynomial $f$ is of type $A_{\mu}$.
The category $\HMF^{gr}_{S}(f)$ is completely reconstructed from the \textit{Auslander-Reiten (AR)}
quiver of the Dynkin quiver of type $A_{\mu}$ described in Figure~\ref{ARquiveforA}.
In fact, the category $\HMF^{gr}_{S}(f)$ admits a full strongly exceptional collection 
$(E_1,\dots,E_{\mu})$ as in Figure~\ref{ARquiveforA} and the \textit{Serre functor} $\mathcal{S}$.
The exsitence of the Serre functor implies that there are no morphisms from $E$ to a object
outside the rectangle whose opposite vertices are given by $E$ and $\mathcal{S}E$ in Figure~\ref{ARquiveforA}.

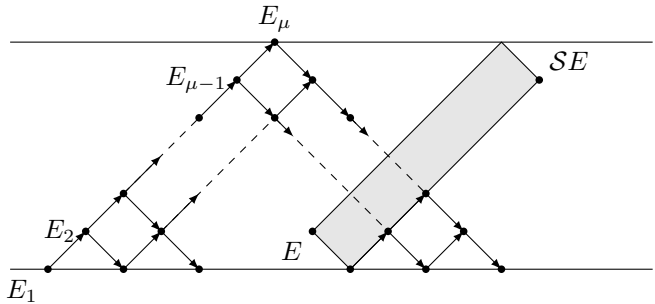
\begin{figure}[H]
  \begin{tikzpicture}
    \pathTriangle[fill=white]{0,0}{3,3}{6,0};
    \pathTriangle[fill=white]{4,3}{7,0}{10,3};
    \draw[fill=gray!20!white] (3.5,0.5)--(4,0)--(6.5,2.5)--(6,3)--cycle;
    \draw (0,0)[below left,font=\small] node{$E_1$};
    \draw (0.5,0.5)[left,font=\small] node{$E_2$};
    \draw (2.5,2.5)[left,font=\small] node{$E_{\mu-1}$};
    \draw (3,3)[above,font=\small] node{$E_{\mu}$};
    \draw (3.5,0.5) [below left,font=\small] node{$E$};
    \draw (6.5,2.5) [above right,font=\small] node{$\mathcal{S}E$};
    \foreach \x/\y in{0/0,1/0,2/0,0.5/0.5,1.5/0.5,1/1} \fill (\x,\y) circle[radius=0.5mm];
    \foreach \x in{0,0.5} \draw[->,>=latex] (\x,\x) to (\x+0.5,\x+0.5);
    \foreach \x in{0.5,1} \draw[->,>=latex] (\x,\x) to (\x+0.5,\x-0.5);
    \draw[->,>=latex] (1,0) to (1.5,0.5);
    \draw[->,>=latex] (1.5,0.5) to (2,0);
    \foreach \x/\y in{0/0,1/0,2/0,0.5/0.5,1.5/0.5,1/1} \fill (\x+2,\y+2) circle[radius=0.5mm];
    \foreach \x in{2,2.5} \draw[->,>=latex] (\x,\x) to (\x+0.5,\x+0.5);
    \foreach \x in{2.5,3} \draw[->,>=latex] (\x,\x) to (\x+0.5,\x-0.5);
    \draw[->,>=latex] (3,2) to (3.5,2.5);
    \draw[->,>=latex] (3.5,2.5) to (4,2);
    \foreach \x/\y in{0/0,1/0,2/0,0.5/0.5,1.5/0.5,1/1} \fill (\x+4,\y) circle[radius=0.5mm];
    \draw[->,>=latex] (4,0) to (4.5,0.5);
    \draw[->,>=latex] (4.5,0.5) to (5,1);
    \draw[->,>=latex] (4.5,0.5) to (5,0);
    \draw[->,>=latex] (5,1) to (5.5,0.5);
    \draw[->,>=latex] (5,0) to (5.5,0.5);
    \draw[->,>=latex] (5.5,0.5) to (6,0);
    \draw[dashed] (1,1) -- (2,2);
    \draw[dashed] (4,2) -- (5,1);
    \draw[dashed] (1.5,0.5) -- (3,2);
    \draw[dashed] (3,2) -- (4.5,0.5);
    \draw (-0.5,0) -- (8,0);
    \draw (-0.5,3) -- (8,3);
    \fill (3.5,0.5) circle[radius=0.5mm];
    \fill (6.5,2.5) circle[radius=0.5mm];
    \draw[->,>=latex] (1,1) to (1.5,1.5);
    \draw[->,>=latex] (1.5,0.5) to (2,1);
    \draw[->,>=latex] (3,2) to (3.25,1.75);
    \draw[->,>=latex] (4,2) to (4.25,1.75);
  \end{tikzpicture}
  \caption{The AR quiver and the Serre functor describing the category $\HMF^{gr}_{S}(f)$.}
  \label{ARquiveforA}
\end{figure}

Classical mirror symmetry is formulated as an isomorphism between 
Frobenius structures arising from Gromov-Witten theory and deformation theory, respectively.
It is expected that the HMS reproduces the classical mirror symmetry.
More precisely, a caterory is believed to reproduce a Frobenius manifold as a space of
deformations of the category with suitable structure on it.
However, there is no formulation of deformations of categories which reproduces Frobenius manifolds. 

Motivated by this expectation, we aim to understand how to formulate deformations of categories.
In this paper, we discuss the triangulated category of graded matrix factorizations
of a deformation of a polynomial $f(x,y,z)$ of type $A_{\mu}$ as a first step.
More precisely, we construct a family of categories of graded matrix factorizations,
which is parametrized by $\C^{\mu}$, associated to a holomorphic function
\[\tilde{f}(x,y,z,t_1,\dots,t_{\mu}): \C^{3}\times \C^{\mu} \rightarrow \C\]
which gives the \textit{semi-universal deformation} of a hypersurface: $V_0=\{f(x,y,z)=0\}$. 
In \cite{kas1972versal}, A. Kas and Schlessinger constructed the semi-universal deformation
of a complete intersection $V_0$ with an isolated singularity: 
\textit{deformation} means a flat map $\pi:\mathcal{X}\rightarrow B$ for which 
$\pi^{-1}(0)=V_0$ and a map $\pi$ is \textit{semi-universal} if any other deformation is induced from $\pi$.
For an isolated hypersurface singularity, the semi-universal deformation is ginen by 
\[ \mathcal{X}=\Big\{ \tilde{f}(x,y,z,t_1,\dots,t_{\mu})=f(x,y,z)+ \sum_{i=1}^{\mu}t_iP_i(x,y,z)=0\Big\}\]
and a projection $\pi:\C^{3} \times \C^{\mu}\rightarrow \C^{\mu}$ where the $P_i$ determine
a $\C$-basis of the \textit{Jacobi ring}
\[ \Jac(f):=\C[x,y,z]\Big/ \Big(\frac{\partial f}{\partial x},\frac{\partial f}{\partial y},
\frac{\partial f}{\partial z}\Big)\]
and $\mu$ is the Milnor number of $f$.
We would like to consider the category of graded matrix factorizations associated to $\tilde{f}$.
However, the holomorphic function $\tilde{f}$ is not weighted homogeneous with respect to
the variables $x,y$ and $z$.
Then, we introduce a formal variable $q$ to make the holomorphic function $\tilde{f}$ homogeneous.
For a polynomial $f(x,y,z)$ of type $A_{\mu}$ and a fixed deformation parameter $(t_1,\dots,t_{\mu})$,
we consider the category of graded matrix factorizations of the weighted homogeneous polynomial given by
\[\tilde{f}_q(x,y,z,q) = x^{\mu+1} + yz + \sum_{i=1}^{\mu} t_i x^{\mu-i} q^{i+1}. \]
with respect to the variables $x,y,z$ and $q$ instead of the holomorphic function $\tilde{f}$.

Our main theorem (Theorem~\ref{mainthm}) is the following.

\begin{thm}
  Let $f \in \C[x,y,z]$ be a polynomial of type $A_{\mu}$.
  For a \textit{generic} (Definition~\ref{genericpara}) parameter,
  the category of graded matrix factorizations of $\tilde{f_q}$ 
  admits a full strongly exceptional collection 
  $(E_1,\dots,E_{\mu},E_{\mu+1},...,E_{2\mu})$
  consisting of graded matrix factorizations of rank $1$.
\end{thm}

\noindent In Figure~\ref{deformation}, we describe the full strongly exceptional collection 
in Theorem~\ref{mainthm} and the Serre functor.
In the category of graded matrix factorizations of $\tilde{f_q}$,
all graded matrix factorizations of rank $1$ forms the similar structure as in Figure~\ref{ARquiveforA}.
However, in contrast to the case of $ADE$ singularities,
starting from graded matrix factorizations of rank $1$, we obtain infinitely many indecomposable
graded matrix factorizations of higher rank by translations and taking mapping cones.

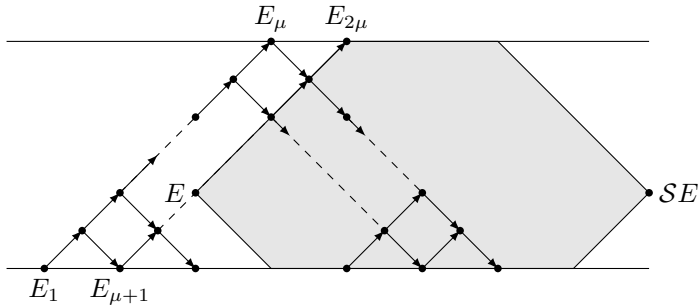
\begin{figure}[H]
  \begin{tikzpicture}
    \pathTriangle[fill=white]{0,0}{3,3}{6,0};
    \pathTriangle[fill=white]{4,3}{7,0}{10,3};
    \draw[fill=gray!20!white] (2,1)--(3,0)--(7,0)--(8,1)--(6,3)--(4,3)--cycle;
    \draw (0,0)[below,font=\small] node{$E_1$};
    \draw (1,0)[below,font=\small] node{$E_{\mu+1}$};
    \draw (3,3)[above,font=\small] node{$E_{\mu}$};
    \draw (4,3)[above,font=\small] node{$E_{2\mu}$};
    \draw (2,1) [left,font=\small] node{$E$};
    \draw (8,1) [right,font=\small] node{$\mathcal{S}E$};
    \foreach \x/\y in{0/0,1/0,2/0,0.5/0.5,1.5/0.5,1/1} \fill (\x,\y) circle[radius=0.5mm];
    \foreach \x in{0,0.5} \draw[->,>=latex] (\x,\x) to (\x+0.5,\x+0.5);
    \foreach \x in{0.5,1} \draw[->,>=latex] (\x,\x) to (\x+0.5,\x-0.5);
    \draw[->,>=latex] (1,0) to (1.5,0.5);
    \draw[->,>=latex] (1.5,0.5) to (2,0);
    \foreach \x/\y in{0/0,1/0,2/0,0.5/0.5,1.5/0.5,1/1} \fill (\x+2,\y+2) circle[radius=0.5mm];
    \foreach \x in{2,2.5} \draw[->,>=latex] (\x,\x) to (\x+0.5,\x+0.5);
    \foreach \x in{2.5,3} \draw[->,>=latex] (\x,\x) to (\x+0.5,\x-0.5);
    \draw[->,>=latex] (3,2) to (3.5,2.5);
    \draw[->,>=latex] (3.5,2.5) to (4,2);
    \foreach \x/\y in{0/0,1/0,2/0,0.5/0.5,1.5/0.5,1/1} \fill (\x+4,\y) circle[radius=0.5mm];
    \draw[->,>=latex] (4,0) to (4.5,0.5);
    \draw[->,>=latex] (4.5,0.5) to (5,1);
    \draw[->,>=latex] (4.5,0.5) to (5,0);
    \draw[->,>=latex] (5,1) to (5.5,0.5);
    \draw[->,>=latex] (5,0) to (5.5,0.5);
    \draw[->,>=latex] (5.5,0.5) to (6,0);
    \draw[dashed] (1,1) -- (2,2);
    \draw[dashed] (4,2) -- (5,1);
    \draw[dashed] (1.5,0.5) -- (3,2);
    \draw[dashed] (3,2) -- (4.5,0.5);
    \draw (-0.5,0) -- (8,0);
    \draw (-0.5,3) -- (8,3);
    \fill (2,1) circle[radius=0.5mm];
    \fill (8,1) circle[radius=0.5mm];
    \fill (4,3) circle[radius=0.5mm];
    \draw[->,>=latex] (3.5,2.5) to (4,3);
    \draw[->,>=latex] (1,1) to (1.5,1.5);
    \draw[->,>=latex] (3,2) to (3.25,1.75);
    \draw[->,>=latex] (4,2) to (4.25,1.75);
  \end{tikzpicture}
  \caption{The structure of all graded matrix factorizations of rank $1$ and the Serre functor.}
  \label{deformation}
\end{figure}

For a weighted homogeneous polynomial, the category of graded matrix factorizations 
is triangulated equivalent to the category of graded maximal Cohen-Macaulay modules 
over the corresponding hypersurface.
In \cite{iyama2013tilting}, Iyama-R. Takahashi showed that categories of graded maximal Cohen-Macaulay modules
over isolated hypersurface singularities admit the Serre functor in the sense of Bondal-Kapranov \cite{bondal1990representable}.
Although the Serre functor on the category of graded matrix factorizations of $\tilde{f_q}$
follows from Theorem 1.5 in \cite{iyama2013tilting}, 
we provide a direct proof using the method due to Bondal \cite{Bondalrep}
as a corollary of our main theorem.

There remain two problems in our study.
For a generic parameter, the holomorphic function $\tilde{f}$ has the following properties:
\begin{itemize}
  \item The holomorphic function $\tilde{f}$ has $\mu$ distinct critical points whose critical values
  are also distinct.
  \item At each critical point, the holomorphic function $\tilde{f}$ defines the 
  $A_{1}\textrm{-}$singularity.
\end{itemize}
The first problem is how to understand these geometric properties from the viewpoint of 
a deformation of the category of graded matrix factorizations.
The second problem is to construct a generator
of the category of graded matrix factorizations of $\tilde{f_q}$
for a \textit{non generic} parameter.
In this case, we need to consider categories of graded matrix factorizations
for non-isolated hypersurface singularities.
We leave these problems for future work.

\subsection{Strategy of proof}
Our strategy of the proof of Theorem~\ref{mainthm} is based on 
the \textit{category generating lemma} which is due to \cite{kajiura2009triangulated}
(see Theorem 4.5).  
Our proof consists of two parts. The first part is checking that 
the collection is a strongly exceptional collection. 
For this purpose, we introduce the notion of \textit{phase} of graded matrix factorizations
and determine the space of morphisms between two graded matrix factorizations of rank $1$ explicitly.
The second part is applying Theorem~\ref{categorygenerating} to the collection.
We need to check that the full triangulated subcategory generated by the collection is closed
under the autoequivalence functor called \textit{grading shift} and contains some special objects.

\subsection{Structure of this paper}
In Section~\ref{prel}, we review the categories of graded matrix factorizations
for weighted homogeneous polynomials and the category of graded maximal
Cohen-Macaulay modules over hypersurfaces.
In Section~\ref{cgt}, after recalling the notion of exceptional collections
on a triangulated category, we review the generation of categories of graded
matrix factorizations which is necessary to prove Theorem~\ref{mainthm}.
In Section~\ref{mainr}, we consider the category of graded matrix factorizations
associated with the semi-universal deformation of the $A_\mu$-singularity
and state our main results.  
Section~\ref{prf} is devoted to proving Theorem~\ref{mainthm}.

\subsection{Acknowledgements.}
I am deeply grateful to my supervisor Professor Hiroshige Kajiura for 
valuable comments. I would like to thank Professor Koji Nishida
for his helpful discussions on the topic of maximal Cohen-Macaulay modules. 
The author is supported by JST SPRING, Grant Number JPMJSP2109.
\section{Preliminaries}\label{prel}

In this section, we set up several definitions which are used in the present paper.
For a positive integer $h$, let $(R, \mathfrak{m})$ be a commutative Noetherian 
$(2\Z/h)$-graded local ring of the Krull dimension $d$ with $R_0=\C$ and the
maximal ideal $\mathfrak{m}:=\oplus_{t\in\frac{2}{h}\Z}R_t$. 
By a graded $R$-module, we always mean a graded $R$-module $M$ which decomposes into the 
direct sum $M=\oplus_{t\in\frac{2}{h}\Z}M_t$. For two graded $R$-modules $M$ and $N$,
a graded $R$-homomorphism $\phi$ of degree $t \in 2\Z/h$ is an $R$-homomorphism
$\phi:M \rightarrow N$ such that $\phi(M_s) \subset N_{s+t}$ for any $s \in 2\Z/h$.

\begin{dfn}
  We denote by $\gr\text{-}R$ the category of finitely generated graded $R$-modules,
  in which morphisms are $R$-homomorphisms of degree zero. The degree shift of 
  $M \in \gr\text{-}R$, denoted by $\tau M$, is defined by $(\tau M)_{s}:=M_{s+\frac{2}{h}}$. 
  This $\tau$ induces an auto-equivalence functor on $\gr\text{-}R$, which we denote by
  the same symbol $\tau$.
\end{dfn}

For two graded $R$-modules $M,N$, we have 
\begin{equation}\label{gradedtoungradedExt}
  \Ext^{i}_R(M,N)\simeq \oplus_{n\in \Z}\Ext^{i}_{\gr\text{-}R}(\tau^{-n}M,N)\simeq 
  \oplus_{n\in \Z}\Ext^{i}_{\gr\text{-}R}(M,\tau^{n}N)
\end{equation}
since $R$ is noetherian. In particular, for $i=0$, 
\begin{equation}\label{gradedtoungradedHom}
  \Hom_R(M,N)\simeq \oplus_{n\in \Z}\Hom_{\gr\text{-}R}(\tau^{-n}M,N)\simeq 
  \oplus_{n\in \Z}\Hom_{\gr\text{-}R}(M,\tau^{n}N)
\end{equation}
forms a graded $R$-module, where the grading of each homogeneous piece is defined 
as $2n/h$. We also have that any graded projective module is free since $R$ is
finitely generated over $R_0=\C$.

\begin{dfn}
  For a graded $R$-module $M \in \gr\text{-}R$, we consider a long exact sequence
  \[0\rightarrow N \rightarrow F_{n-1} \rightarrow F_{n-2} \rightarrow \cdots
  \rightarrow F_1 \rightarrow F_0 \rightarrow M\rightarrow 0\]
  in $\gr\text{-}R$ where each $F_i$ is graded free. \textit{The reduced $n$-th syzygy} 
  $\syz^{n}(M)$ of $M$ is the graded $R$-module obtained from $N$ by deleting 
  all graded free summands. 
\end{dfn}

The reduced $n$-th syzygy $\syz^{n}(M)$ is uniquely determined by $M$ and $n$
up to isomorphism.

\subsection{Category of graded maximal Cohen-Macaulay modules.}

In this subsection, we recall some notions of the category of graded
maximal Cohen-Macaulay modules over a graded Gorenstein local ring $R$.
We refer to \cite{Matsumura1987} and \cite{yoshino1990maximal} for terminologies
and the statements presented here.

\begin{dfn}
  An object $M\in \gr\text{-}R$ is called a \textit{graded maximal Cohen-Macaulay module} over $R$
  if $\Ext^{i}_R(R/\mathfrak{m},M)=0$ for $i<d$.
  We denote by $\CM^{gr}(R)$ the full subcategory of $\gr\text{-}R$ consisting of
  all graded maximal Cohen-Macaulay modules which forms an extension-closed subcategory.
\end{dfn}

An object $K_R \in \gr\text{-}R$ is called a \textit{canonical module} of $R$
if $\Ext^{i}_R(R/\mathfrak{m},K_R)=0$ for $i \neq d$ and
$\Ext^{d}_R(R/\mathfrak{m},K_R) \simeq \C$.
A cononical module need not exist, but if it dose, it is unique up to isomorphism.

\begin{dfn}
  For a graded Gorenstein local ring $R$, the cononical module $K_R$ is isomorphic
  to $\tau^{-\varepsilon(R)}R$ for some $\varepsilon(R) \in \Z$. The integer $\varepsilon(R)$
  is called the \textit{Gorenstein parameter} of $R$.
\end{dfn}

\begin{lem}
  The category $\CM^{gr}(R)$ is a Frobenius category, namely,
  it has enough projectives and enough injectives and the projectives
  coincide with the injectives.
\end{lem}

\begin{dfn}
  We define an additive category $\underline{\CM^{gr}}(R)$ as follows: 
  its objects are graded maximal Cohen-Macaulay module over $R$ and,
  for any $M,N \in \underline{\CM^{gr}}(R)$,
  the space of morphisms $\underline{\Hom}_{\gr\text{-}R}(M,N)$ is given by 
  $\Hom_{\gr\text{-}R}(M,N)/\mathcal{P}(M,N)$,
  where $\mathcal{P}(M,N)$ is the subspace consisting of morphisms factoring through 
  projectives, namely, $g\in \mathcal{P}(M,N)$ 
  if and only if $g=g'' g'$ for $g':M \rightarrow P$ and
  $g'':P \rightarrow N$ with a projective object $P$.
\end{dfn}

Since $\underline{\CM^{gr}}(R)$ is the stable category of the Frobenius category
$\CM^{gr}(R)$, due to Happel \cite{happel1988triangulated}, one obtains that: 

\begin{pro}
  The stable category $\underline{\CM^{gr}}(R)$ forms a triangulated category.
\end{pro}

In the rest of this paper, we consider the case that $R$ is a graded Gorenstein 
local ring of hypersurface $S/(f)$.
A polynomial $f(x_1,\dots,x_n)\in S:=\C[x_1,\dots,x_n]$ is called \textit{weighted homogeneous},
if there are positive integers $w_1,\dots,w_n$ and $h$ such that 
\[f(\lambda^{w_1} x_1,\dots,\lambda^{w_n} x_n)=\lambda^{h}f(x_1,\dots,x_n)\]
for $\lambda \in \C^{*}$. 
A pair $(w_1,\dots,w_n;h)$ of positive integers  with $\gcd(w_1,\dots,w_n)=1$ is called 
a \textit{weight system}. 
For a weight system, we define the \textit{Euler vector field} $E$ by
\[ E:= \sum_{i=1}^{n}\frac{w_i}{h} x_i \frac{\partial}{\partial x_i}.\]
For a weight system, the ring $S$ becomes a $(2\Z/h)$-graded local ring with the maximal ideal 
$\mathfrak{m}=(x_1,\dots,x_n)$ by putting $\deg(x_i)=2w_i/h$. 
A graded piece decomposition $S=\oplus_{t\in\frac{2}{h}\Z_{\ge 0}}S_t$ is given by
$S_t:=\{g\in S\ |\ 2Eg=tg\}$ with $f \in S_2$.
For the category $\underline{\CM^{gr}}(R)$, we remark the following properties.

\begin{pro}\label{syztau}
  For a graded maximal Cohen-Macaulay module $M$ over $R$, we have
  \[\syz^{2}(M) \simeq \tau^{-h}M\]
  in $\underline{\CM^{gr}}(R)$.
\end{pro}

\begin{pf}
  Let $M$ be a graded maximal Cohen-Macaulay module over $R$.
  We have the following graded free resolution of $M$ as a graded $S$-module:
  \[ 0 \rightarrow S^n \overset{f_0}{\rightarrow} S^n \rightarrow M \rightarrow 0\]
  for some $n\in \Z_{>0}$.
  We consider the following commutative diagram in $\gr\text{-}S$.
  \[\begin{tikzcd}
	& 0 & 0 & {\tau^{-h}M} \\
	0 & {\tau^{-h}S^n} & {\tau^{-h}S^n} & {\tau^{-h}M} & 0 \\
	0 & {S^n} & {S^n} & M & 0 \\
	& {R^n} & {R^n} & M
	\arrow[from=1-2, to=2-2]
	\arrow[from=1-3, to=2-3]
	\arrow["{\textrm{id}}"', from=1-4, to=2-4]
	\arrow[from=2-1, to=2-2]
	\arrow["{f_0}", from=2-2, to=2-3]
	\arrow["f"', from=2-2, to=3-2]
	\arrow["\pi", from=2-3, to=2-4]
	\arrow["f"', from=2-3, to=3-3]
	\arrow[from=2-4, to=2-5]
	\arrow["f"', from=2-4, to=3-4]
	\arrow[from=3-1, to=3-2]
	\arrow["{f_0}", from=3-2, to=3-3]
	\arrow[from=3-2, to=4-2]
	\arrow["\pi", from=3-3, to=3-4]
	\arrow[from=3-3, to=4-3]
	\arrow[from=3-4, to=3-5]
	\arrow["{\textrm{id}}"', from=3-4, to=4-4]
  \end{tikzcd}\]
  By the snake lemma, we have the exact sequence 
  \[ 0 \rightarrow \tau^{-h}M \rightarrow R^n \rightarrow R^n \rightarrow M \rightarrow 0\]
  in $\gr\text{-}R$. By the definition of syzygy, we have $\syz^{2}(M) \simeq \tau^{-h}M$. 
\end{pf}

Proposition~\ref{syztau} implies that the isomorphism $T^{-2}M \simeq \tau^{-h}M$ holds in
$\underline{\CM^{gr}}(R)$ where $T$ is the translation functor on $\underline{\CM^{gr}}(R)$.

\begin{pro}\label{period}
  A graded maximal Cohen-Macaulay module over a graded local ring of a hypersurface 
  has a periodic graded free resolution with periodicity two.
\end{pro}

We assume that $R$ is a graded isolated singularity, namely,
the graded localization $R_{\mathfrak{p}}$ is regular for any garded prime
$\mathfrak{p} \neq \mathfrak{m}$.
The category $\underline{\CM^{gr}}(R)$ admits the \textit{Serre functor}:

\begin{thm}[{{\cite[Theorem 1.5]{iyama2013tilting}}\cite[Theorem 3.8]{kajiura2009triangulated}}]\label{cmserre}
  Let $R$ be a graded Gorenstein local ring of a hypersurface which defines an isolated singularity
  and has the canonical module.
  The functor $\mathcal{S}:=T^{d-1}\tau^{-\varepsilon(R)}$ is the Serre functor
  on $\underline{\CM^{gr}}(R)$.
  More precisely, the functor $\mathcal{S}$ is an auto-equivalence functor
  satisfying the following properties:
  \begin{itemize}
    \item[(i)] $\Hom_{\gr\text{-}R}(M,\mathcal{S}M) \simeq \C$
    for any object $M \in \underline{\CM^{gr}}(R)$.
    \item[(ii)] The isomorphism of (i) induces the following 
    nondegenerate bilinear map:
    \[ \Hom_{\gr\text{-}R}(M,N) \otimes \Hom_{\gr\text{-}R}(N,\mathcal{S}M)
    \rightarrow \C\]
    for any $M,N \in \underline{\CM^{gr}}(R)$. 
  \end{itemize}
\end{thm}

\subsection{Category of graded matrix factorizations.}

In this subsection, we recall some notions of the category of graded matrix factorizations
associated to a weighted homogeneous polynomial $f$. 
We refer to \cite{kajiura2007matrix} and \cite{kajiura2009triangulated} 
for terminologies and the statements presented here.

\begin{dfn}
  For a weighted homogeneous polynomial $f \in S$, we define an additive category $\MF^{gr}_{S}(f)$
  as follows. The set of objects consist of all \textit{graded matrix factorizations} of $f$ defined by
  \[ \overline{F}=\Bigl(\mf{F_0}{f_0}{f_1}{F_1}\Bigr),\]
  where $F_0$ and $F_1$ are graded free $S$-modules of finite rank, and
  $f_0:F_0 \rightarrow F_1$ is a graded $S$-homomorphism of degree zero and
  $f_1:F_1\rightarrow F_0$ is a graded $S$-homomorphism of degree two such that 
  $f_1 f_0 = f \cdot \mathrm{id}_{F_0}$ and $f_0 f_1 = f \cdot \mathrm{id}_{F_1}$.
  A morphism $\Phi : \overline{F} \rightarrow \overline{F'}$ in $\MF^{gr}_{S}(f)$ is a 
  pair of graded $S$-homomorphisms
  $\phi_0 : F_0 \rightarrow F'_0$ and $\phi_1 : F_1 \rightarrow F'_1$ of degree zero
  satisfying $\phi_1 f_0 = f'_0 \phi_0$ and $\phi_0 f_1 = f'_1 \phi_1$.
  The space of all morphisms from $\overline{F}$ to $\overline{F'}$ is
  denoted by $\MF^{gr}_{S}(f)(\overline{F},\overline{F'})$.
\end{dfn}

For a graded matrix factorization $\overline{F}$, the rank of $F_0$
coincides with that of $F_1$, which we call the rank of the matrix factorization 
$\overline{F}$ and denote it by $\rk (\overline{F})$.
An object $\overline{F}\in \HMF^{gr}_{S}(f)$ is a zero object if and only if 
it is a direct sum of the graded matrix factorizations of the forms
\[ \Bigl( \mf{\tau^{n}S}{f}{1}{\tau^{n+h}S} \Bigr) \]
and
\[ \Bigl( \mf{\tau^{n'}S}{1}{f}{\tau^{n'}S} \Bigr) \]
for some $n,n' \in \Z$.
A graded matrix factorization $\overline{F}$ is called \textit{reduced} if it has no 
direct summand of zero object. 

\begin{dfn}
  An additive category $\HMF^{gr}_{S}(f)$ is defined by the following data.
  The set of objects is given by the set of all graded matrix factorizations:
  \[ \Ob(\HMF^{gr}_{S}(f)):= \MF^{gr}_{S}(f).\]
  For any two objects $\overline{F},\overline{F'} \in \MF^{gr}_{S}(f)$, the space of morphisms
  is given by the quotient module:
  \[ \HMF^{gr}_{S}(f)(\overline{F},\overline{F'}):=\MF^{gr}_{S}(f)(\overline{F},\overline{F'})/\sim,\]
  where two morphisms $\Phi,\Phi' : \overline{F} \rightarrow \overline{F'}$ is called 
  homotopic $\Phi \sim \Phi'$ if there are graded $S$-homomorphisms $h_0 : F_0 \rightarrow F'_1$ 
  of degree minus two and $h_1 : F_1 \rightarrow F'_0$ of degree zero such that 
  \[\Phi - \Phi'=( h_1 f_0 + f'_1 h_0,f'_0 h_1 + h_0 f_1).\]
\end{dfn}

The auto-equivalence functor $\tau$ on $\gr\text{-}R$ induces the one on
$\HMF^{gr}_{S}(f)$, which we also denote by the same symbol $\tau$.
Explicitly, the action of $\tau$ is given by 
\[\tau\overline{F}:=\Bigl(\mf{\tau F_0}{\tau(f_0)}{\tau(f_1)}{\tau F_1}\Bigr),\]
for an object $\overline{F}$ and 
$\tau\Phi:=(\tau(\phi_0),\tau(\phi_1))$ for a morphism $\Phi=(\phi_0,\phi_1)$. 
We shall denote by $T$ the translation functor on $\HMF^{L_{f}}_{S}(f)$, 
which is given by 
\[ T\overline{F}:=\Bigl(\mf{F_{1}}{-f_1}{-\tau^{h}(f_0)}{F_{0}}\Bigr)\]
for an object $\overline{F}$ and $T\Phi:=(\phi_1,\tau^{h}(\phi_0))$ for a morphism $\Phi=(\phi_0,\phi_1)$.

Next, we explain the triangulated structure in $\HMF^{L_{f}}_{S}(f)$.

\begin{dfn}
  For a morphism $\Phi=(\phi_0,\phi_1) \in \HMF^{gr}_{S}(f)(\overline{F},\overline{F'})$, 
  we define the \textit{mapping cone} $C(\Phi) \in \MF^{gr}_{S}(f)$ as
  \[ C(\Phi):=\Bigl(\mf{F_{1}\oplus F'_0}{c_0}{c_1}{\tau^{h}F_0 \oplus F'_1}\Bigr),\]
  where
  \[ c_0:=\begin{pmatrix}
          -f_1   & 0    \\
          \phi_1 & f'_0 \\
           \end{pmatrix}
    ,\  
      c_1:=\begin{pmatrix}
          -\tau^{h}(f_0)   & 0 \\
          \tau^{h}(\phi_0) & f'_1 \\
           \end{pmatrix}.\]
\end{dfn}

Since there exist morphisms $\overline{F'} \overset{^{t}(0\ \mathrm{id})}{\rightarrow} C(\Phi)$ and 
$C(\Phi) \overset{(-\mathrm{id}\ 0)}{\rightarrow} T\overline{F}$, one can easily see that

\begin{pro}
  Each exact triangle in $\HMF^{gr}_{S}(f)$ is isomorphic to a triangle of the form
  \[ \overline{F} \overset{\Phi}{\rightarrow} \overline{F'} 
  \overset{^{t}(0\ \mathrm{id})}{\rightarrow} C(\Phi) \overset{(-\mathrm{id}\ 0)}{\rightarrow} T\overline{F}\]
  for some $\overline{F},\overline{F'} \in \HMF^{gr}_{S}(f)$ and 
  $\Phi \in \HMF^{gr}_{S}(f)(\overline{F},\overline{F'})$.
\end{pro}

From these descriptions of auto-equivalences $\tau$ and $T$, we obtain the following.

\begin{pro}\label{T2}
  On the category $\HMF^{gr}_{S}(f)$, we have $T^2 = \tau^{h}$.
\end{pro}

In this paper, for the convenience of the later discussion, we introduce the notion
of a \textit{phase} for a graded matrix factorization.
Let $\overline{F}=\Bigl(\mf{F_0}{f_0}{f_1}{F_1}\Bigr)$ 
be a reduced graded matrix factorization of rank $r$. 
We choose homogeneous free basis $(b_1,\dots,b_r;\bar{b}_1,\dots,\bar{b}_r)$
such that $F_0 = b_1 S \oplus \dots \oplus b_r S$ and  
$F_1 = \bar{b}_1 S \oplus \dots \oplus \bar{b}_r S$.
Then, we can express $F_0$ and $F_1$ as $\tau^{t_1} S \oplus \dots \oplus \tau^{t_r} S$
and $\tau^{\bar{t}_1} S \oplus \dots \oplus \tau^{\bar{t}_r} S$, respectively,
such that $b_i \in S_{2t_i/h}$ and $\bar{b}_i \in S_{2\bar{t}_i/h}$, $i=1,\dots,r$.
The graded $S$-homomorphisms $f_0:F_0 \rightarrow F_1$ and $f_1:F_1 \rightarrow F_0$
have the matrix expressions $q_0$ and $q_1$, respectively, whose entries are homogeneous elements in $S$. 
They are defined as $f_0(b_1,\dots,b_r)=(b_1,\dots,b_r)q_0$ and 
$f_1(\bar{b}_1,\dots,\bar{b}_r)=(\bar{b}_1,\dots,\bar{b}_r)q_1$.
The pair $(t_1,\dots,t_r;\bar{t}_1,\dots,\bar{t}_r)$ of integers is uniquely
determined by the matrix expressions $q_0$ and $q_1$ up to the action of $\tau$.

\begin{dfn}[cf. {\cite[Definition 6.1 and Definition 6.2]{kajiura2009triangulated}}]\label{phasespectrum}
  For an indecomposable graded matrix factorization $\overline{F} \in \HMF^{gr}_{S}(f)$,  
  the \textit{phase} $\phi(\overline{F})$ of $\overline{F}$ is defined by 
  \[ \phi(\overline{F}):= \frac{1}{\rk(\overline{F})}\sum_{i=1}^{\rk(\overline{F})}(t_i+\bar{t}_i). \]
  For two indecomposable graded matrix factorizations 
  $\overline{F},\overline{F'} \in \HMF^{gr}_{S}(f)$, we define
  $$\phi(\overline{F},\overline{F'}):= \phi(\overline{F'})-\phi(\overline{F}).$$
  The spectrum $\mathfrak{sp}(\overline{F},\overline{F'})$ is defined by the 
  following multi-set of rational numbers:
  \[ \mathfrak{sp}(\overline{F},\overline{F'}):=
  \{ \phi(\overline{F},\tau^{n}\overline{F'})^
  {\dim_{\C}\HMF^{gr}_{S}(f)(\overline{F},\tau^{n}\overline{F'})}\ |\ n\in\Z\},\]
  where $\phi(\overline{F},\tau^{n}\overline{F'})^
  {\dim_{\C}\HMF^{gr}_{S}(f)(\overline{F},\tau^{n}\overline{F'})}$ indicates
  that we include $\dim_{\C}\HMF^{gr}_{S}(f)(\overline{F},\tau^{n}\overline{F'})$ copies
  of $\phi(\overline{F},\tau^{n}\overline{F'})$ in $\mathfrak{sp}(\overline{F},\overline{F'})$.
  In particular, we denote ${\mathfrak sp}(\overline{F},\overline{F})$ by $\mathfrak{sp}(\overline{F})$.
\end{dfn}

From Definition \ref{phasespectrum}, for any two indecomposable graded matrix factorizations
$\overline{F},\overline{F'}$, one has
\begin{align*}
  &\phi(\tau\overline{F})= \phi(\overline{F})+2, \\ 
  &\phi(T\overline{F})= \phi(\overline{F})+h, \\
  &\mathfrak{sp}(\overline{F},\tau\overline{F'})=\mathfrak{sp}(\tau\overline{F},\overline{F'})=
  \mathfrak{sp}(\overline{F},\overline{F'}), \\
  &\mathfrak{sp}(T\overline{F},T\overline{F'})
  =\mathfrak{sp}(\overline{F},\overline{F'}).
\end{align*}

At the end of this section, we give the construction of a graded matrix factorization
of $f$ for an $R$-module $M$ of the form $S/I$, where the ideal $I$ is generated by 
a regular sequence of $R$ and contains $f$.

\begin{dfn}[cf. {\cite[Section 2.3]{dyckerhoff2011compact}}]
  Let $M$ be an $R$-module of the form $S/(g_1,\dots,g_s)$ such that $g_i \in S_{2 t_i/h}$
  for some $t_i \in \Z$, and $g_1,\dots,g_s$ forms a regular sequence and $f \in (g_1,\dots,g_s)$.
  We write $f = g_1 h_1 + \cdots + g_s h_s$ with $h_i \in S_{2-2 t_i/h}$ and put 
  $P:=\oplus_{i=1}^{s}\tau^{-t_i}S$.
  The Koszul resolution of $M$ as a graded $S$-module
  \[ 0\rightarrow \wedge^{s}P \overset{d_s}{\rightarrow} \wedge^{s-1}P 
  \overset{d_{s-1}}{\rightarrow} \cdots \overset{d_{2}}{\rightarrow} \wedge^1 P 
  \overset{d_{1}}{\rightarrow} \wedge^0 P=S \rightarrow M \rightarrow 0\]
  yeilds the graded matrix factorization 
  $\overline{F}=\Bigl(\mf{F_0}{f_0}{f_1}{F_1}\Bigr)$ of $f$, where
  \[ F_0:=\bigoplus_{k} \tau^{kh}(\wedge^{2k}P)\ ,\ 
  F_1:=\bigoplus_{k}\tau^{kh}(\wedge^{2k-1}P),\]
  and graded $S$-homomorphisms $f_0$ and $f_1$ are determined by $d_{\bullet}$ and 
  $\bar{d}_{\bullet}:\wedge^{\bullet-1}P \rightarrow \wedge^{\bullet}P$
  which gives a ``contracting homotopy" of $d_{\bullet}$ satisfying $\bar{d}^2=0$ 
  and $d\bar{d}+\bar{d}d=f\cdot \textrm{id}_{\wedge^{\bullet}P}$. 
  It is called the \textit{stabilization} of $M$ and denoted by $M^{stab}$.
\end{dfn}

\begin{exmp}
  We construct the stabilization $M^{stab}$ of $M$ in the case of $s=2$ explicitly.
  Let $M$ be an $R$-module of the form $S/(g_1,g_2)$.
  We write $f=g_1 h_1 + g_2 h_2$. The Koszul resolution of $M$ is given by
  \[ 0 \rightarrow \tau^{(-t_1-t_2)}S \overset{d_2}{\rightarrow} 
  \tau^{-t_1}S \oplus \tau^{-t_2}S \overset{d_1}{\rightarrow} S \rightarrow M \rightarrow 0,\]
  where $d_1(x_1,x_2)=g_1 x_1 - g_2 x_2$ and $d_2(y)=(g_2 y,g_1 y)$ 
  for homogeneous elements $x_1,x_2$ and $y$ in $S$.
  A contracting homotopy of such $d_{\bullet}$ is given by
  $\bar{d}_1(z)=(h_1 z,-h_2 z)$ and $\bar{d}_2(x_1,x_2)=h_2 x_1 + h_1 x_2$.
  We can easily see that $\bar{d}_1 d_1 + d_2 \bar{d}_2 = f \cdot \mathrm{id}_{F_1}$
  and $d_1 \bar{d}_1 + \bar{d}_2 d_2 = f \cdot \mathrm{id}_{F_0}$, 
  then we obtain $M^{stab}$ which has the form
  \[ M^{stab} = \Bigl(\mf{S \oplus \tau^{(h-t_1-t_2)}S}{f_0}{f_1} 
  {\tau^{(h-t_1)}S \oplus \tau^{(h-t_2)}S}\Bigr),\]
  where
  \[ f_0=\begin{pmatrix}
        h_1   &  g_2 \\
        -h_2  &  g_1 \\
         \end{pmatrix},
    f_1=\begin{pmatrix}
        g_1   &  -g_2\\
        h_2   &  h_1 \\
        \end{pmatrix}.
  \]
\end{exmp}
\section{Category generating theorem.}\label{cgt}

In this section, we recall the generation of the category $\HMF^{gr}_{S}(f)$.
We assume that $R$ is a graded Gorenstein local ring of a hypersurface $S/(f)$
which is a graded isolated singularity.
First, we collect some definitions and facts concerning admissible categories and
exceptional collections from \cite{Bondalrep} and \cite{orlov2009derived}.
A $\C$-linear triangulated category $\mathcal{T}$ with the translation functor $T$ 
is called to be of \textit{finite type} if, for any $E,E' \in \mathcal{T}$, 
the space $\Hom_{\mathcal{T}}(E,T^{p}E')$ is of finite rank over $\C$ and 
\[ \sum_{p \in \Z}\dim_{\C}\Hom_{\mathcal{T}}(E,T^{p}E') < \infty.\]

\begin{dfn}
  Let $\mathcal{T}$ be a triangulated category and $\mathcal{T}'$ be a full subcategory
  of $\mathcal{T}$. The right orthogonal to $\mathcal{T}'$ is a full subcategory 
  $(\mathcal{T}')^{\perp}$ of $\mathcal{T}$ consisting of all objects $M$ such that
  $\Hom_{\mathcal{T}}(N,M)=0$ for any $N \in \mathcal{T}'$. 
  A full subcategory $\mathcal{T}'$ is called 
  \textit{right admissible} if, for any $L \in \mathcal{T}$, there exists an exact triangle
  $N \rightarrow L \rightarrow M \rightarrow TN$ with $N \in \mathcal{T}'$ and
  $M \in (\mathcal{T}')^{\perp}$.
\end{dfn}

\begin{dfn}
  Let $\mathcal{T}$ be a $\C$\textrm{-}linear triangulated category of finite type
  with the translation functor $T$.
  \begin{itemize}
    \item[(i)] An object $E$ in $\mathcal{T}$ is called \textit{exceptional} if
    $\Hom_{\mathcal{T}}(E,E)=\C \cdot \mathrm{id}_{E}$ and 
    $\Hom_{\mathcal{T}}(E,T^{p}E)=0$ when $p \neq 0$.
    \item[(ii)] An ordered set $\mathcal{E}=(E_1,\dots,E_l)$ of exceptional objects
    is called an \textit{exceptional collection} if $\Hom_{\mathcal{T}}(E_m,T^{p}E_n)=0$
    for all $p \in \Z$ and $m > n$.
    \item[(iii)] An exceptional collection $\mathcal{E}=(E_1,\dots,E_l)$ is called 
    a \textit{strongly exceptional collection} if $\Hom_{\mathcal{T}}(E_m,T^{p}E_n)=0$
    for all $p \neq 0$ and $m,n=1,\dots,l$.
    \item[(iv)] An exceptional collection $\mathcal{E}=(E_1,\dots,E_l)$ is called
    \textit{full} if the smallest full triangulated subcategory of $\mathcal{T}$ containing
    all elements in $\mathcal{E}$ is equivalent to $\mathcal{T}$ as a triangulated
    category.
  \end{itemize}
\end{dfn}

\begin{pro}[{\cite[Theorem 3.2]{Bondalrep}}]
  Let $\mathcal{T}$ be a $\C$\textrm{-}linear triangulated category of finite type
  and $\mathcal{T}'$ be a full triangulated subcategory generated 
  by an exceptional collection $(E_1,\dots,E_l)$. Then $\mathcal{T}'$ is right
  admissible.
\end{pro}

Next, we recall the equivalence between the category of graded maximal Cohen-Macaulay 
modules over $R$ and the category of graded matrix factorizations of $f$. 

\begin{pro}[{\cite[Corollary 6.3]{eisenbud1980homological}}]\label{A}
  The category $\HMF^{gr}_{S}(f)$ is equivalent to $\underline{\CM^{gr}}(R)$ 
  as a triangulated category. The equivalence is given by the functor
  \[\overline{F}=\Bigl(\mf{F_0}{f_0}{f_1}{F_1}\Bigr) \mapsto M:=\Cker(f_1).\]
\end{pro}

The following theorem is necessary to prove our main result.

\begin{thm}[{\cite[Theorem 4.5]{kajiura2009triangulated}}{\cite[Lemma 2.12]{aramaki2020maximally}}]
  \label{categorygenerating}
  Let $\mathcal{T}'$ be a right admissible full triangulated subcategory of 
  $\HMF^{gr}_{S}(f)$, with $R$ is a graded isolated singularity, satisfying the following conditions:
  \begin{itemize}
    \item[(i)] The grading shift functor $\tau$ on $\HMF^{gr}_{S}(f)$
    induces an auto-equivalence of $\mathcal{T}'$,
    \item[(ii)] $\mathcal{T}'$ has an object which is isomorphic to $(R/\mathfrak{m})^{stab}$
    in $\HMF^{gr}_{S}(f)$. 
  \end{itemize}
  Then, the natural fully faithful functor $\mathcal{T}' \rightarrow \HMF^{gr}_{S}(f)$
  is a triangulated equivalence.
\end{thm}

We remark that the proof of Theorem~\ref{categorygenerating} relies on the equivalence in Proposition~\ref{A}.
Moreover, by Theorem~\ref{cmserre}, Proposition~\ref{A} and Theorem~\ref{categorygenerating}, 
we obatin the following corollary.

\begin{cor}
  The category $\HMF^{gr}_{S}(f)$ admits the Serre functor $T^{d-1}\tau^{-\varepsilon(R)}$.
\end{cor}
\section{Main results.}\label{mainr}

In this section, we shall only consider a weighted homogeneous polynomial which defines
the $A_{\mu}$-singularity at the origin in three variables:
\[ f(x,y,z) = x^{\mu + 1} + yz.\]
For a hypersurface $V_0:=\{f(x,y,z)=0\}$, the semi-universal deformation of $V_0$ is given
by the holomorphic function
\[ \tilde{f}(x,y,z,t_1,\dots,t_\mu) = x^{\mu + 1} + yz + \sum_{i=1}^{\mu} t_i x^{\mu-i} ,\]
where $t_1,\dots,t_\mu$ are coordinates on the base space $\C^{\mu}$ of the deformation.
It is known that the base space $\C^{\mu}$ of the semi-universal deformation associated to $f$
can be identified with $\mathfrak{H}/W$, where $\mathfrak{H}$ is a subspace of $\C^{\mu+1}$ defined by
\[ \mathfrak{H}:=\Bigl\{ {\mathbf s}=(s_1,\dots,s_{\mu+1})\in \C^{\mu+1}\ |
\ \sum_{i=1}^{\mu+1} s_i=0\Bigr\}\]
and $W$ is the Weyl group of the root system of type $A_{\mu}$.
The action of the Weyl group $W$  on $\mathfrak{H}$ is a permutation of coordinates.
We have the isomorphism $\mathfrak{H}/W \rightarrow \C^{\mu}$ given by
\[ {\mathbf s} \mapsto (I_2({\mathbf s}),...,I_{\mu+1}({\mathbf s})),\]
where $I_{k}({\mathbf s})$ is an elementary symmetric polynomial of degree $k$ in 
$(\mu+1)$-variables $s_1,\dots,s_{\mu+1}$ for all $k=2,\dots,\mu+1$.

From now on, we define the category of garaded matrix factorizations associated to the holomorphic function
$\tilde{f}$. The weighted homogeneous polynomial $f$ has the weight system
$(1,b,\mu+1-b;\mu+1)$, $1 \le b \le \mu$. However, in general, 
the holomorphic function $\tilde{f}$ is not weighted homogeneous with respect to the variables
$x,y$ and $z$. Now, we introduce the formal variable $q$ with degree $2/h$ and consider
a weighted homogeneous polynomial associated to $\tilde{f}$ which is given by 
\[\tilde{f}_q(x,y,z,q) = x^{\mu+1} + yz + \sum_{i=1}^{\mu} t_i x^{\mu-i} q^{i+1}. \]
Note that the polynomial $\tilde{f_q}$ can be factored as follows:
\[\tilde{f_q}(x,y,z,q) = \prod_{i=1}^{m}(x+s_i q)^{n_i} + yz\]
where $s_i \neq s_j$ for all $i \neq j$ and $\sum_{i=1}^{m}n_i=\mu+1$.

\begin{dfn}\label{genericpara}
  We call a parameter $(t_1,\dots,t_\mu) \in \C^{\mu}$ \textit{generic} 
  if $(s_1,\dots,s_{\mu+1}) \in \mathfrak{H}$
  corresponding to $(t_1,\dots,t_\mu)$ through the isomorphism 
  $\mathfrak{H}/W \rightarrow \C^{\mu}$ is in the configuration space
  \[ \mathrm{Conf}_{\mu+1}(\C) = \{ (s_1,\dots,s_{\mu+1}) \in \C^{\mu+1}\ |
  \ s_i \neq s_j (i \neq j )\}.\]
\end{dfn}

In the rest of this paper, we focus on the category of graded matrix factorizations of $\tilde{f_q}$
for a \textit{generic} parameter and provide the details of its triangulated structure.
For the sake of simplicity in what follows,
we recall the following result which is called \textit{Kn\"{o}rrer periodicity}.

\begin{pro}[{\cite[Theorem 3.1]{Knorrerperiod}}]\label{Knorrer}
  Let $R$ be a hypersurface $S/(f)$ and set $R_2:=S[y,z]/(f+yz)$.
  We heve an equivalence between the triangulated categories $\HMF^{gr}_{S}(f)$ and
  $\HMF^{gr}_{S[y,z]}(f+yz)$.
\end{pro}

By Proposition \ref{Knorrer}, putting $S:=\C[x,q]$ and 
\[\tilde{f}_q(x,q) = x^{\mu+1} + \sum_{i=1}^{\mu} t_i x^{\mu-i} q^{i+1},\]
we consider the category $\HMF^{gr}_{S}(\tilde{f_q})$ instead of the category 
$\HMF^{gr}_{S[y,z]}(\tilde{f_q}+yz)$.

Before stating our main result, we would like to emphasize that,
due to the formal variable $q$ introduced for the deformation,
the corresponding hypersurface $S/(\tilde{f_q})$ defines a graded isolated singularity
when a deformation parameter $(t_1,\dots,t_\mu)$ is generic.

\begin{pro}
  The Gorenstein local ring $S/(\tilde{f_q})$ is a graded isolated singularity
  if and only if a parameter $(t_1,\dots,t_\mu)$ is generic.
\end{pro}

\begin{proof}
  We shall prove both directions. First, we prove the necessity.
  It suffices to prove the contrapositive.
  We suppose that a parameter is not generic.
  There exists a number $k \in \{1,...,\mu+1\}$ such that
  \[\tilde{f_q}(x,y,z,q) = (x+s_k q)^2 g(x,q)\]
  where $g(x,q)$ is a weighted homogeneous polynomial. The partial derivatives are:
  \[\frac{\partial \tilde{f_q}}{\partial x}=(x+s_k q)\Big\{ 2g(x,q)+(x+s_k q)\frac{\partial g}{\partial x} \Big\}\]
  and
  \[\frac{\partial \tilde{f_q}}{\partial q}=(x+s_k q)\Big\{ 2s_k g(x,q)+(x+s_k q)\frac{\partial g}{\partial q} \Big\}.\]
  At any point $(x_0,q_0)$ on the line defined by $x+s_k q=0$, all partial derivatives vanish.
  Thus, the origin is not an isolated singularity. 
  
  Conversely, we suppose that a parameter is generic.
  By Definition~\ref{genericpara}, the polynomial $\tilde{f_q}$ can be factored as follows:
  \[\tilde{f_q}(x,y,z,q) = \prod_{i=1}^{\mu+1}(x+s_i q)\]
  where $s_i \neq s_j$ for all $i \neq j$.
  By Euler's identity of $\tilde{f_q}$, we have:
  \begin{equation}\label{Euleridnt}
    (\mu+1)\tilde{f_q}(x, q)= x\frac{\partial \tilde{f_q}}{\partial x} + q\frac{\partial \tilde{f_q}}{\partial q}.
  \end{equation}
  This implies that any point where the partial derivatives vanish automatically satisfies $\tilde{f_q} = 0$.
  If there exists a singular point $(x_0,q_0)\neq 0$, by the equation~\ref{Euleridnt},
  the point $(x_0,q_0)$ must lie on at least one of the lines we denote by $x+s_{i_0} q=0$.
  Because the partial derivatives satisfy
  \[\frac{\partial \tilde{f_q}}{\partial x}\Big|_{(x_0,q_0)}
  =\Big\{
  \prod_{j \neq i_0}(x+s_j q)+ (x+s_{i_0} q)\Big(\prod_{j \neq i_0}(x+s_j q)\Big)'
  \Big\}\Big|_{(x_0,q_0)}=0\]
  and
  \[\frac{\partial \tilde{f_q}}{\partial q}\Big|_{(x_0,q_0)}
  =\Big\{
  s_{i_0}\prod_{j \neq i_0}(x+s_j q)+ (x+s_{i_0} q)\Big(\prod_{j \neq i_0}(x+s_j q)\Big)'
  \Big\}\Big|_{(x_0,q_0)}=0,\]
  then we have
  \[\prod_{j \neq i_0}(x_0+s_j q_0)=0.\]
  This implies that $x_0+s_j q_0=0$ for some other index $j \neq i_0$.
  This means that the point $(x_0,q_0)$ is an intersection of two distinct lines.
  Therefore, no singular point exists other than origin.
\end{proof}

\subsection{Graded matrix factorizations of rank $1$}

In this subsection, we give graded matrix factorizations of $\tilde{f_q}$ of rank $1$ explicitly.
Given a subset $I$ of $\{1,\dots,\mu+1\}$, we can define the graded matrix 
factorization $\overline{F_{I}}$ of rank $1$ as follows:
\[ \overline{F_{I}}:=\Bigl(\mf{S}{f_0}{f_1}{\tau^{|I|}S}\Bigr),\]
where 
\[ f_0=\prod_{i\in I}(x+s_iq)\ ,\ f_1=\prod_{j\notin I}(x+s_jq)\]
and $|I|$ is the number of elements of $I$. Note that 
$\overline{F_{I}}$ becomes a zero object in $\HMF^{gr}_{S}(\tilde{f_q})$
if $I$ is an empty set or $\{1,\dots,\mu+1\}$ itself.
We immediately obtain the following lemma.

\begin{lem}
  There exists an one\textrm{-}to\textrm{-}one correspondence between
  the set of all graded matrix factorizations of rank $1$ up to the grading shift and
  the power set of $\{1,\dots,\mu+1\}$ except for an empty set and $\{1,\dots,\mu+1\}$.
\end{lem}

We need some explicit data of morphisms between two graded matrix factorizations 
of rank $1$. For this purpose, we caluculate the spectrum
$\mathfrak{sp}(\overline{F_I},\overline{F_J})$ defined 
in Definition \ref{phasespectrum}. 

\begin{lem}\label{I=J}
  We suppose that $I=J$. We represent the spectrum $\mathfrak{sp}(\overline{F_I})$ as 
  $\{ p_0 \le \dots \le p_k\}$ for some $k \in \Z_{\ge 0}$. Then, one has
  \[ 0 = p_0 \le \dots \le p_k = 2(\mu-1),\]
  for any subset $I \subsetneq \{1,\dots,\mu+1\}$.
\end{lem}

\begin{proof} 
  It is enough to show that the space of morphisms 
  $\HMF^{gr}_{S}(\tilde{f_q})(\overline{F_{I}},\tau^{n}\overline{F_{I}})$
  is zero for $n < 0$ and $ n \ge \mu$.
  Since the part of negative degree of $S$ is zero, we have
  $\HMF^{gr}_{S}(\tilde{f_q})(\overline{F_{I}},\tau^{k}\overline{F_{I}})=0$ 
  for the case $n < 0$.
  For the case $n \ge \mu$, it is enough to show 
  $\HMF^{gr}_{S}(\tilde{f_q})(\overline{F_{I}},\tau^{\mu}\overline{F_{I}})=0$.
  Any morphism $\Phi \in \HMF^{gr}_{S}(\tilde{f_q})(\overline{F_{I}},\tau^{\mu}\overline{F_{I}})$
  is given by the form $(g,g)$, where $g$ is a homogeneous element with 
  degree $2\mu/h$ in $S$.
  By the assumption, we have $\gcd(f_0(x,1),f_1(x,1))=1$. Then, there exist 
  polynomials $a$ and $b$ in $\C[x]$ such that 
  $af_0(x,1)+bf_1(x,1)=1$ from Bézout's identity. 
  Dividing $a\cdot g(x,1)$ and $b \cdot g(x,1)$ by $f_1(x,1)$ and $f_0(x,1)$ respectively, 
  we have $a \cdot g(x,1)=f_1(x,1)p+r$ and $b \cdot g(x,1)=f_0(x,1)q+s$, 
  where $p,q$ are the quotients and $r,s$ are the remainders. 
  We have 
  \[ g(x,1) = a \cdot g(x,1) f_0(x,1) + b \cdot g(x,1) f_1(x,1),\]
  \[(p+q)f_0(x,1)f_1(x,1)=g(x,1)-rf_0(x,1)-sf_1(x,1).\]
  Since the highest degree of the polynomial of the RHS is equal to $\mu$,
  we have $p+q=0$. Then, we have $g(x,1)=rf_0(x,1)+sf_1(x,1)$ which implies that 
  any morphism $\Phi \in \HMF^{gr}_{S}(\tilde{f_q})(\overline{F_{I}},\tau^{\mu}\overline{F_{I}})$
  is homotopic to the zero morphism.
\end{proof}

\begin{lem}\label{IsubsetneqJ}
  We suppose that $I \subsetneq J$. We represent the spectrum 
  $\mathfrak{sp}(\overline{F_I},\overline{F_J})$ as
  $\{ p_0 \le \dots \le p_k\}$ for some $k \in \Z_{\ge 0}$. Then, one has
  \[ |J|-|I| = p_0 \le \dots \le p_k = |J|-|I| + 2(\mu-1).\]
  In particular, we have the morphism
  \[ \Phi_{I}^{J}:=\Big(1,\prod_{i\in I^{c} \cap J}(x+s_i q)\Big)
  \in \HMF^{gr}_{S}(\tilde{f_q})(\overline{F_{I}},\overline{F_{J}}).\]
\end{lem}

\begin{proof}
  It is proved by the direct caluculations of 
  $\HMF^{gr}_{S}(\tilde{f_q})(\overline{F_{I}},\tau^{n}\overline{F_{J}})$
  based on the explicit form of $\overline{F_{I}}$ and $\overline{F_{J}}$. 
\end{proof}

\begin{lem}\label{IsupsetneqJ}
  We suppose that $I \supsetneq J$. We represent the spectrum 
  $\mathfrak{sp}(\overline{F_I},\overline{F_J})$ as 
  $\{ p_0 \le \dots \le p_k\}$ for some $k \in \Z_{\ge 0}$. Then, one has
  \[ |I|-|J| = p_0 \le \dots \le p_k = |I|-|J| + 2(\mu-1).\]
  In particular, we have the morphism 
  \[\overline{\Phi}_{I}^{J}:=\Big(\prod_{i\in I \cap J^{c}}(x+s_i q),1\Big)
  \in \HMF^{gr}_{S}(\tilde{f_q})(\overline{F_{I}},\tau^{|I|-|J|}\overline{F_{J}}).\]
\end{lem}

\begin{proof}
  For two subsets $I,J$ satisfying $I \supsetneq J$,
  by the translation functor $T$ on $\HMF^{gr}_{S}(\tilde{f_q})$,
  one has $T\overline{F_{I}} \simeq \tau^{|I|}\overline{F_{I^c}}$.
  Then we have the isomorphism
  \begin{equation}\label{subsetsub}
    \begin{split}
      \HMF^{gr}_{S}(\tilde{f_q})(\overline{F_{I}},\tau^{n}\overline{F_{J}}) 
      &\simeq 
      \HMF^{gr}_{S}(\tilde{f_q})(T\overline{F_{I}},T\tau^{n}\overline{F_{J}})\\
      &\simeq 
      \HMF^{gr}_{S}(\tilde{f_q})(\tau^{|I|}\overline{F_{I^c}},\tau^{n+|J|}\overline{F_{J^c}})\\
      &\simeq 
      \HMF^{gr}_{S}(\tilde{f_q})(\overline{F_{I^c}},\tau^{n+|J|-|I|}\overline{F_{J^c}}).
    \end{split}
  \end{equation}
  The morphism $\overline{\Phi}_{I}^{J}$ is obtained by the isomorphism (\ref{subsetsub})
  for $n=|I|-|J|$, which is corresponding to $\Phi_{I^c}^{J^c}$ in Lemma~\ref{IsubsetneqJ}.
  Therefore the statement follows from Lemma~\ref{IsubsetneqJ} and the isomorphism (\ref{subsetsub}).
\end{proof}

\begin{lem}\label{IJempty}
  We suppose that $I \cap J = \emptyset$. Then, we have
  $\mathfrak{sp}(\overline{F_I},\overline{F_J})=\emptyset$.
\end{lem}

\begin{proof}
  It is proved by the direct caluculations of 
  $\HMF^{gr}_{S}(\tilde{f_q})(\overline{F_{I}},\tau^{n}\overline{F_{J}})$
  based on the explicit form of $\overline{F_{I}}$ and $\overline{F_{J}}$.
\end{proof}

\subsection{Full exceptional collections of the category  $\HMF^{gr}_{S}(\tilde{f_q})$}

In this subsection, we give the full strongly exceptional collection 
consisting of some graded matrix factorizations of rank $1$
in $\HMF^{gr}_{S}(\tilde{f_q})$.

\begin{lem}\label{shifthom}
  For any $I \subsetneq \{1,\dots,\mu+1\}$, we have
  $\mathfrak{sp}(\overline{F_I},T\overline{F_I})=\emptyset$.
\end{lem}

\begin{proof}
  This statement immediately follows from 
  $T\overline{F_{I}} \simeq \tau^{|I|}\overline{F_{I^c}}$ and Lemma \ref{IJempty}.
\end{proof}

By Proposition \ref{T2}, Lemma \ref{I=J} and Lemma \ref{shifthom}, 
we obtain the following corollary.

\begin{cor}
  For any subset $I \subsetneq \{1,\dots,\mu+1\}$, the graded matrix factorization
  $\overline{F_{I}}$ is an exceptional object in $\HMF^{gr}_{S}(\tilde{f_q})$.
\end{cor}

The following is the main result of the present paper.

\begin{thm}\label{mainthm}
  Let $a=(a_1,\dots,a_{\mu}) \in \{ 1,\dots,\mu+1\}^{\mu}$ be integers satisfying $a_i \neq a_j$ for $i \neq j$.
  We set a subset 
  $$I_k:=\{ a_1,\dots,a_k\} \subset \{1,\dots,\mu+1\}$$ 
  for $k=1,\dots,\mu$
  and define an ordered collection 
  $\mathcal{E}_{a}=(E_1,\dots,E_{\mu},E_{\mu+1},\dots,E_{2\mu})$
  as follows:
  \begin{equation*}
    E_k:=
    \begin{cases*}
      \overline{F_{I_k}} \text{\ if\ }1 \le k \le \mu,\\
      \tau\overline{F_{I_{k-\mu}}}\ \text{if}\ \mu+1 \le k \le 2\mu.
    \end{cases*}
  \end{equation*}
  Then, an ordered collection $\mathcal{E}_{a}$ forms a full strongly exceptional collection in 
  $\HMF^{gr}_{S}(\tilde{f_q})$.
\end{thm}

We prove Theorem~\ref{mainthm} in Section~\ref{prf}. 
As a consequence of Theorem~\ref{mainthm}, we obtain the following.

\begin{cor}\label{Serre}
  The functor $\tau^{\mu-1}$ is the Serre functor on 
  $\HMF^{gr}_{S}(\tilde{f_q})$. 
\end{cor}

Remark that the exsistence of the Serre functor on $\HMF^{gr}_{S}(\tilde{f_q})$ is followed
by Theorem~\ref{cmserre} and Proposition~\ref{A}.
More precisely, it is a special case of $d=1$ and $-\varepsilon(R)=\mu-1$ in Theorem~\ref{cmserre}.
In this paper, we give the proof of Corollary~\ref{Serre} by the direct caluculations and using
the following result due to Bondal \cite{Bondalrep}.

\begin{pro}\label{ser}
  Let $\mathcal{T}$ be a $\C$-linear triangulated category  
  of finite type which admits some full exceptional collection $(E_1,\dots,E_l)$.
  For a object $X\in \mathcal{T}$, we define the object $L^{l-1}X$ by induction:
  $L^0 X =X$, and 
  \begin{equation}\label{D}
    L^{k+1}X \rightarrow \Hom_{\mathcal{T}}(E_{l-k},L^{k}X)\otimes E_{l-k}
    \rightarrow L^{k}X\rightarrow TL^{k+1}X,
  \end{equation}
  where $\Hom_{\mathcal{T}}(E,F)\otimes E$ indicates 
  $\oplus_{p\in \Z} (T^{-p}E)^{\oplus \dim_{\C} \Hom_{\mathcal{T}}(E,T^{p}F)}$.
  Then, $\mathcal{T}$ has the Serre functor $\mathcal{S}$ and the isomorphism
  $\mathcal{S}(E_i) \simeq T^{l-1}L^{l-1}(E_i)$ holds for all $i=1,\dots,l$.
\end{pro}

\begin{proof}[Proof of Corollary \ref{Serre}]
  We will apply Proposition \ref{ser} to the full strongly exceptional collection
  $\mathcal{E}_{a}$ described in Theorem \ref{mainthm}.
  We assume $a_i=i$ for all $i=1,\dots,\mu$ without loss of generality.
  It is enough to show $\tau^{\mu-1}E_{2\mu} \simeq T^{2\mu-1}L^{2\mu-1}(E_{2\mu})$.
  If $k=1$, the exact triangle (\ref{D}) is
  \[ L^{1}(E_{2\mu}) \rightarrow E_{2\mu-1}
  \rightarrow E_{2\mu} \rightarrow \tau^{\mu}\overline{F_{\{\mu\}}}.\]
  Then, we have $L^{1}(E_{2\mu})\simeq \overline{F_{\{\mu\}^c}}$.
  By the direct caluculations, we have
  \[\HMF^{gr}_{S}(\tilde{f_q})(T^{-p}\overline{E_{2\mu-k}},\overline{F_{\{\mu\}^c}})=0\]
  for any $p\in \Z$ and $k=2,...,\mu$.
  Therefore, the exact triangle (\ref{D}) is 
  \[ L^{k}(E_{2\mu}) \rightarrow 0
  \rightarrow L^{k-1}(E_{2\mu}) \rightarrow L^{k-1}(E_{2\mu})\]
  for $k=2,...,\mu$. If $k=\mu+1$, the exact triangle (\ref{D}) is
  \[ L^{\mu+1}(E_{2\mu}) \rightarrow T^{-(\mu-1)}E_{2\mu-1}
  \rightarrow L^{\mu}(E_{2\mu})\simeq T^{-(\mu-1)}\overline{F_{\{\mu\}^c}}
  \rightarrow T^{-(\mu-1)}\tau^{\mu-1}\overline{F_{\{\mu+1\}}}.\]
  Then, we have $L^{\mu+1}(E_{2\mu})\simeq T^{-(\mu-1)}\tau^{-1}\overline{F_{\{\mu+1\}^c}}$.
  By Lemma~\ref{IsubsetneqJ}, we have
  \[\HMF^{gr}_{S}(\tilde{f_q})(T^{-p}\overline{E_{2\mu-k}},\tau^{-1}\overline{F_{\{\mu+1\}^c}})=0\]
  for any $p\in \Z$ and $k=\mu+2,...,2\mu-1$.
  Therefore, the exact triangle (\ref{D}) is
  \[ L^{k}(E_{2\mu}) \rightarrow 0
  \rightarrow L^{k-1}(E_{2\mu}) \rightarrow L^{k-1}(E_{2\mu})\]
  for $k=\mu+2,...,2\mu-1$. Then, we have 
  \[T^{2\mu-1}L^{2\mu-1}(E_{2\mu})\simeq 
  T^{2\mu-1}T^{-(2\mu-3)}\tau^{-1}\overline{F_{\{\mu+1\}^c}}
  =\tau^{\mu}\overline{F_{\{\mu+1\}^c}}=\tau^{\mu-1}E_{2\mu}.\]
\end{proof}
\section{Proof of Theorem \ref{mainthm}}\label{prf}

First, we will show that the collection $\mathcal{E}_{a}$ described in Theorem~\ref{mainthm}
forms a strongly exceptional collection in $\HMF^{gr}_{S}(\tilde{f_q})$.

\begin{lem}
  The collection $\mathcal{E}_{a}$ forms an exceptional collection in $\HMF^{gr}_{S}(\tilde{f_q})$.
\end{lem}

\begin{proof}
We fix $m,n \in \{1,2,\dots,2\mu-1,2\mu\}$ satisfying $m>n$. 
We should check following five cases:
\begin{itemize}
  \item[(1)]  For any $m,n \in \{1,\dots,\mu\}$, we consider the case with
  \[\HMF^{gr}_{S}(\tilde{f_q})(E_m,T^{p}E_n)= 
  \HMF^{gr}_{S}(\tilde{f_q})(\overline{F_{I_m}},T^{p}\overline{F_{I_n}}).\]
  If $p= 2l +1$, $l \in \Z$, we have 
  \[\HMF^{gr}_{S}(\tilde{f_q})(\overline{F_{I_m}},T^{p}\overline{F_{I_n}})
  =\HMF^{gr}_{S}(\tilde{f_q})(\overline{F_{I_m}},\tau^{l(\mu+1)+|I_n|}\overline{F_{I^{c}_n}})=0\]
  by the direct caluculations of 
  $\HMF^{gr}_{S}(\tilde{f_q})(\overline{F_{I_m}},\tau^{|I_n|}\overline{F_{I^{c}_n}})$.

  If $p=2l$, $l\in \Z$, by Lemma \ref{IsupsetneqJ}, one has
  $\phi(\overline{F_{I_m}},\tau^{l(\mu+1)}\overline{F_{I_n}}) \notin 
  \mathfrak{sp}(\overline{F_{I_m}},\overline{F_{I_n}})$.
  Then, we have 
  \[\HMF^{gr}_{S}(\tilde{f_q})(\overline{F_{I_m}},T^{p}\overline{F_{I_n}})= 
  \HMF^{gr}_{S}(\tilde{f_q})(\overline{F_{I_m}},\tau^{l(\mu+1)}\overline{F_{I_n}})=0.\]
  
  \item[(2)] For any $m,n \in \{\mu+1,\dots,2\mu\}$, we consider the cese with
  \[\HMF^{gr}_{S}(\tilde{f_q})(E_m,T^{p}E_n)=
  \HMF^{gr}_{S}(\tilde{f_q})(\tau\overline{F_{I_m}},T^{p}\tau\overline{F_{I_n}}).\]
  By the isomorphism
  \[\HMF^{gr}_{S}(\tilde{f_q})(\tau\overline{F_{I_m}},T^{p}\tau\overline{F_{I_n}})
  \simeq \HMF^{gr}_{S}(\tilde{f_q})(\overline{F_{I_m}},T^{p}\overline{F_{I_n}}),\]
  we have $\HMF^{gr}_{S}(\tilde{f_q})(E_m,T^{p}E_n)=0$ for all $p \in \Z$.

  \item[(3)] For any $m \in \{\mu+1,\dots,2\mu\}$ and $n \in \{1,\dots,\mu\}$
  satisfying $m-\mu = n$, namely, $I_{m-\mu}=I_n$, we consider the case with
  \[\HMF^{gr}_{S}(\tilde{f_q})(E_m,T^{p}E_n)=
  \HMF^{gr}_{S}(\tilde{f_q})(\tau\overline{F_{I_n}},T^{p}\overline{F_{I_n}}).\]
  By the isomorphism
  \[\HMF^{gr}_{S}(\tilde{f_q})(\tau\overline{F_{I_n}},T^{p}\overline{F_{I_n}})
  \simeq \HMF^{gr}_{S}(\tilde{f_q})(\overline{F_{I_n}},T^{p}\tau^{-1}\overline{F_{I_n}})\]
  and Lemma \ref{I=J}, we have $\HMF^{gr}_{S}(\tilde{f_q})(E_m,T^{p}E_n)=0$ for all $p\in \Z$.

  \item[(4)] For any $m \in \{\mu+1,\dots,2\mu\}$ and $n \in \{1,\dots,\mu\}$
  satisfying $m-\mu < n$, namely, $I_{m-\mu} \subsetneq I_n$, we consider the case with
  \[\HMF^{gr}_{S}(\tilde{f_q})(E_m,T^{p}E_n)=
  \HMF^{gr}_{S}(\tilde{f_q})(\tau\overline{F_{I_{m-\mu}}},T^{p}\overline{F_{I_n}}).\]
  If $p= 2l +1$, $l \in \Z$, since the intersection of $I_{m-\mu}$ and $I^{c}_n$ is 
  empty, we have 
  \[\HMF^{gr}_{S}(\tilde{f_q})(\tau\overline{F_{I_{m-\mu}}},T^{p}\overline{F_{I_n}})
  \simeq \HMF^{gr}_{S}(\tilde{f_q})(\overline{F_{I_{m-\mu}}},
  \tau^{l(\mu+1)+|I_n|-1}\overline{F_{I^{c}_n}})=0\]
  by Lemma~\ref{IJempty}.

  If $p=2l$, $l\in \Z$, one obtains that
  $\phi(\overline{F_{I_{m-\mu}}},\tau^{l(\mu+1)-1}\overline{F_{I_n}}) \notin 
  \mathfrak{sp}(\overline{F_{I_{m-\mu}}},\overline{F_{I_n}})$ 
  by Lemma~\ref{IsubsetneqJ}.
  Then, we have 
  \[\HMF^{gr}_{S}(\tilde{f_q})(\tau\overline{F_{I_{m-\mu}}},T^{p}\overline{F_{I_n}})
  \simeq \HMF^{gr}_{S}(\tilde{f_q})(\overline{F_{I_{m-\mu}}},
  \tau^{l(\mu+1)-1}\overline{F_{I_n}})=0.\]
  
  \item[(5)] For any $m \in \{\mu+1,\dots,2\mu\}$ and $n \in \{1,\dots,\mu\}$
  satisfying $m-\mu > n$, namely, $I_n \subsetneq I_{m-\mu}$, we consider the case with
  \[\HMF^{gr}_{S}(\tilde{f_q})(E_m,T^{p}E_n)=
  \HMF^{gr}_{S}(\tilde{f_q})(\tau\overline{F_{I_{m-\mu}}},T^{p}\overline{F_{I_n}}).\]
  If $p= 2l +1$, $l \in \Z$, we have 
  \[\HMF^{gr}_{S}(\tilde{f_q})(\tau\overline{F_{I_{m-\mu}}},T^{p}\overline{F_{I_n}})
  \simeq \HMF^{gr}_{S}(\tilde{f_q})(\overline{F_{I_{m-\mu}}},
  \tau^{l(\mu+1)+|I_n|-1}\overline{F_{I^{c}_n}})=0\]
  by the direct caluculations of 
  $\HMF^{gr}_{S}(\tilde{f_q})
  (\overline{F_{I_{m-\mu}}},\tau^{l(\mu+1)+|I_n|-1}\overline{F_{I^{c}_n}})$.
  If $p=2l$, $l\in \Z$, we have 
  \[\HMF^{gr}_{S}(\tilde{f_q})(\tau\overline{F_{I_{m-\mu}}},T^{p}\overline{F_{I_n}})
  \simeq \HMF^{gr}_{S}(\tilde{f_q})(\overline{F_{I_{m-\mu}}},
  \tau^{l(\mu+1)-1}\overline{F_{I_n}})=0\]
  since one obtains that
  $\phi(\overline{F_{I_{m-\mu}}},\tau^{l(\mu+1)-1}\overline{F_{I_n}}) \notin 
  \mathfrak{sp}(\overline{F_{I_{m-\mu}}},\overline{F_{I_n}})$
  by Lemma~\ref{IsupsetneqJ}. 
\end{itemize}
\end{proof}

\begin{lem}
  The collection $\mathcal{E}_{a}$ forms a strongly exceptional collection in $\HMF^{gr}_{S}(\tilde{f_q})$.
\end{lem}

\begin{proof}
We fix $m,n \in \{1,2,\dots,2\mu-1,2\mu\}$ satisfying $m<n$ and $p \neq 0$.
We should check following five cases:

\begin{itemize}
  \item[(1)']  For any $m,n \in \{1,\dots,\mu\}$, namely, $I_{m} \subsetneq I_{n}$,
  we consider the case with
  \[\HMF^{gr}_{S}(\tilde{f_q})(E_m,T^{p}E_n)= 
  \HMF^{gr}_{S}(\tilde{f_q})(\overline{F_{I_m}},T^{p}\overline{F_{I_n}}).\]
  If $p= 2l +1$, $l \in \Z$, we have 
  \[\HMF^{gr}_{S}(\tilde{f_q})(\overline{F_{I_m}},T^{p}\overline{F_{I_n}})
  =\HMF^{gr}_{S}(\tilde{f_q})(\overline{F_{I_m}},\tau^{l(\mu+1)+|I_n|}\overline{F_{I^{c}_n}})=0\]
  by the direct caluculations of 
  $\HMF^{gr}_{S}(\tilde{f_q})(\overline{F_{I_m}},\tau^{|I_n|}\overline{F_{I^{c}_n}})$.

  If $p=2l$, $l\in \Z$, one obtains that
  $\phi(\overline{F_{I_m}},\tau^{l(\mu+1)}\overline{F_{I_n}}) \notin 
  \mathfrak{sp}(\overline{F_{I_m}},\overline{F_{I_n}})$ by Lemma \ref{IsubsetneqJ}.
  Then, we have 
  \[\HMF^{gr}_{S}(\tilde{f_q})(\overline{F_{I_m}},T^{p}\overline{F_{I_n}})= 
  \HMF^{gr}_{S}(\tilde{f_q})(\overline{F_{I_m}},\tau^{l(\mu+1)}\overline{F_{I_n}})=0.\]
  
  \item[(2)'] For any $m,n \in \{\mu+1,\dots,2\mu\}$, by the isomorphism
  \[\HMF^{gr}_{S}(\tilde{f_q})(E_i,T^{p}E_j)=
  \HMF^{gr}_{S}(\tilde{f_q})(\tau\overline{F_{I_m}},T^{p}\tau\overline{F_{I_n}})
  \simeq \HMF^{gr}_{S}(\tilde{f_q})(\overline{F_{I_m}},T^{p}\overline{F_{I_n}}),\]
  we have the same results as in the case~(1)'. 
  
  \item[(3)'] For any $m \in \{1,\dots,\mu\}$ and $n \in \{\mu+1,\dots,2\mu\}$
  satisfying $m = n - \mu$, namely, $I_{n-\mu}=I_m$, we consider the case with
  \[\HMF^{gr}_{S}(\tilde{f_q})(E_m,T^{p}E_n)=
  \HMF^{gr}_{S}(\tilde{f_q})(\overline{F_{I_m}},T^{p}\tau\overline{F_{I_m}}).\]
  By Lemma \ref{I=J} and Lemma \ref{shifthom}, we obtain that the RHS is zero
  for any $p\in \Z\setminus \{0\}$.
  
  \item[(4)'] For any $m \in \{1,\dots,\mu\}$ and $n \in \{\mu+1,\dots,2\mu\}$
  satisfying $m > n - \mu$, namely, $I_{n-\mu} \subsetneq I_m$, 
  we consider the case with
  \[\HMF^{gr}_{S}(\tilde{f_q})(E_m,T^{p}E_n)=
  \HMF^{gr}_{S}(\tilde{f_q})(\overline{F_{I_{m}}},T^{p}\tau\overline{F_{I_{n-\mu}}}).\]
  By using a similar caluculation as in the case (5), we obtain that the RHS is zero
  for any $p\in \Z\setminus \{0\}$.
  
  \item[(5)'] For any $m \in \{1,\dots,\mu\}$ and $n \in \{\mu+1,\dots,2\mu\}$
  satisfying $m < n - \mu$, namely, $I_{m} \subsetneq I_{n-\mu}$, 
  we consider the case with
  \[\HMF^{gr}_{S}(\tilde{f_q})(E_m,T^{p}E_n)=
  \HMF^{gr}_{S}(\tilde{f_q})(\overline{F_{I_{m}}},T^{p}\tau\overline{F_{I_{n-\mu}}}).\]
  By using a similar caluculation as in the case (4), we obtain that the RHS is zero
  for any $p\in \Z\setminus \{0\}$.
\end{itemize}
\end{proof}

\subsection{Grading shifts}

Finally we will employ our Theorem \ref{categorygenerating} to show that 
$\mathcal{E}_{a}$ is a full strongly exceptional collection.
For this purpose, we first show that that $\langle\mathcal{E}_{a}\rangle$
satisfies the condition (i) in Theorem \ref{categorygenerating}.

\begin{lem}\label{closed}
  The smallest full triangulated subcategory $\langle\mathcal{E}_{a}\rangle$
  contains all graded matrix factorizations of rank $1$.
\end{lem}

\begin{proof}
  We assume $a_i=i$ for all $i=1,\dots,\mu$ without loss of generality. By Lemma \ref{IsubsetneqJ}
  and Lemma \ref{IsupsetneqJ}, there are exact triangles
  \[\overline{F_{I}} \xrightarrow{\Phi_{I}^{J}} 
  \overline{F_{J}} \rightarrow 
  C(\Phi_{I}^{J}) \rightarrow T\overline{F_{I}}\]
  and
  \[\overline{F_{I'}} \xrightarrow{\overline{\Phi}_{I'}^{J'}} 
  \tau^{|I'|-|J'|}\overline{F_{J'}} \rightarrow 
  C(\overline{\Phi}_{I'}^{J'}) \rightarrow T\overline{F_{I'}}\]
  for $I \subsetneq J$ and $I' \supsetneq J'$.
  Moreover, there exist natural isomorphisms
  \[ C(\Phi_{I}^{J}) \simeq \tau^{|I|}\overline{F_{I^c \cap J}}\ ,
  \ C(\overline{\Phi}_{I'}^{J'}) \simeq \tau^{|I'|-|J'|}\overline{F_{I^c \cup J}}.\]

\begin{cl}\label{I}
  The collection $\mathcal{E}_{a}$ generates 
  the objects $\tau\overline{F_{\{k\}^c}}$ for all $k=1,\dots,\mu+1$.
\end{cl}

  For all $k=2,\dots,\mu$, there is an exact triangle
  \[ \overline{F_{I_k}} \xrightarrow{\overline{\Phi}_{I_k}^{I_{k-1}}}
  \tau\overline{F_{I_{k-1}}} \xrightarrow{} 
  \tau\overline{F_{\{k\}^c}} \xrightarrow{} T\overline{F_{I_k}},\]
  Moreover, we have $\tau\overline{F_{\{1\}^c}} \simeq 
  T\overline{F_{I_1}}$ and 
  $\tau\overline{F_{\{\mu+1\}^c}} \simeq \tau\overline{F_{I_\mu}}$
  for $k=1$ and $k=\mu+1$ respectively.

\begin{cl}\label{II}
  The collection $\mathcal{E}_{a}$ generates 
  the objects $\tau^{\mu}\overline{F_{\{k\}}}$
  for all $k=1,\dots,\mu+1$.
\end{cl}
 
  For $k=1$, we have the following exact triangles
  \[\overline{F_{I_1}} \xrightarrow{\Phi_{I_1}^{I_{\mu}}} 
  \overline{F_{I_{\mu}}} \rightarrow 
  \tau\overline{F_{\{2,\dots,\mu\}}} \rightarrow T\overline{F_{I_1}},\]
  and
  \[\tau\overline{F_{\{2,\dots,\mu\}}} \xrightarrow{\tau\Phi_{\{2,\dots,\mu\}}^{I_{\mu}}} \tau\overline{F_{I_{\mu}}} \rightarrow
  \tau^{\mu}\overline{F_{\{1\}}} \rightarrow T\tau\overline{F_{\{2,\dots,\mu\}}}.\]
  For $k=2,\dots,\mu-1$, we have the following exact triangle
  \[\overline{F_{I_k}} \xrightarrow{\Phi_{I_k}^{I_{\mu}}} \overline{F_{I_{\mu}}} \rightarrow 
  \tau^{k}\overline{F_{\{k+1,\dots,\mu\}}} \rightarrow T\overline{F_{I_k}}.\]
  For $k=2,\dots,\mu$, we also have the following exact triangle
  \[\tau\overline{F_{I_{k-1}}} \xrightarrow{\tau\Phi_{I_{k-1}}^{I_{\mu}}} 
  \tau\overline{F_{I_{\mu}}} \rightarrow 
  \tau^{k}\overline{F_{\{k,\dots,\mu\}}} \rightarrow T\tau\overline{F_{I_{k-1}}}.\]
  Then, we have the exact triangles 
  \[\tau^{k}\overline{F_{\{k+1,\dots,\mu\}}} \rightarrow \tau^{k}\overline{F_{\{k,\dots,\mu\}}}
  \rightarrow \tau^{\mu}\overline{F_{\{k\}}} \rightarrow T\tau^{k}\overline{F_{\{k+1,\dots,\mu\}}}\]
  for $k=2,\dots,\mu-1$. Moreover, we have
  $\tau^{\mu}\overline{F_{\{\mu+1\}}} \simeq T\overline{F_{I_{\mu}}}$
  for $k=\mu+1$. 
  Hence Claim~\ref{II} follows.

\begin{cl}\label{III}
  The collection $\mathcal{E}_{a}$ generates 
  the objects $\tau^{m}\overline{F_{\{1\}}},\dots,\tau^{m}\overline{F_{\{m+1\}}}$
  for $m=1,\dots,\mu-1$.
\end{cl}

  By Claim~\ref{I}, we have the isomorphism
  $\overline{F_{\{l\}}} \simeq T^{-1}\tau\overline{F_{\{l\}^c}}$. Hence 
  $\langle\mathcal{E}_{a}\rangle$ contains $\overline{F_{\{l\}}}$ for all $l=1,\dots,\mu+1$. 
  Then, there is an exact triangle
  \[\overline{F_{\{l\}}} \xrightarrow{\Phi_{\{l\}}^{I_{m}}} \overline{F_{I_{m}}} 
  \rightarrow \tau\overline{F_{I_{m}\setminus \{l\}}} \rightarrow T\overline{F_{\{l\}}}\]
  for all $l=1,\dots,m$ and
  \[\overline{F^{\{l'\}}} \xrightarrow{\Phi_{\{l'\}}^{I_{m+1}}} \overline{F_{I_{m+1}}} \rightarrow 
  \tau\overline{F_{I_{m+1}\setminus \{l'\}}} \rightarrow T\overline{F_{\{l'\}}}\]
  for all $l'=1,\dots,m+1$. If $l'=1,\dots,m$ and $l=l'$, then we have the following
  exact triangle
  \[\tau\overline{F_{I_{m}\setminus \{l\}}} \rightarrow
  \tau\overline{F_{I_{m+1}\setminus \{l'\}}} \rightarrow  
  \tau^{m}\overline{F_{\{m+1\}}} \rightarrow T\tau\overline{F_{I_{m}\setminus \{l\}}}.\]
  If $l'=m+1$, then we have the following exact triangle
  \[\tau\overline{F_{I_{m}\setminus \{l\}}} \rightarrow
  \tau\overline{F_{I_{m}}} \rightarrow \tau^{m}\overline{F_{\{l\}}} 
  \rightarrow T\tau\overline{F_{I_{m}\setminus \{l\}}}\]
  for all $l=1,\dots,m$. Therefore Claim~\ref{III} follows.

\begin{cl}\label{IV}
  The collection $\mathcal{E}_{a}$ generates
  the objects $\overline{F_{\{k_1,\dots,k_l\}^c}}$
  for $l=1,...,\mu-1$ and $k_1 \neq \cdots \neq k_{l}$.
\end{cl}

  By Claim~\ref{II}, we have 
  $\langle\mathcal{E}_{a}\rangle$ contains $\overline{F_{\{k_1\}^c}}
  \simeq T^{-1}\tau^{\mu}\overline{F_{\{k_1\}}}$ for all $k_1=1,\dots,\mu+1$.
  We will show Claim~\ref{IV} for the case $l=2$.
  Considering the case $m=\mu-1$ in Claim~\ref{III}, we have
  $\langle\mathcal{E}_{a}\rangle$ contains the objects
  $\tau^{\mu-1}\overline{F_{\{k\}}}$ for $k=1,\dots,\mu$. 
  In addition, we obtain the object 
  $\tau^{\mu-1}\overline{F_{\{\mu+1\}}}$ by the following exact triangle
  \[\overline{F_{I_{\mu-1}}} \rightarrow \overline{F_{\{\mu\}^c}} 
  \simeq T^{-1}\tau^{\mu}\overline{F_{\{\mu\}}} \rightarrow 
  \tau^{\mu-1}\overline{F_{\{\mu+1\}}} \rightarrow T\overline{F_{I_{\mu-1}}}.\]
  Therefore, we have the following exact triangle 
  \[\overline{F_{\{k_1\}^c}} \rightarrow \tau^{\mu-1}\overline{F_{\{k_2\}}} 
  \rightarrow \tau^{\mu-1}\overline{F_{\{k_1,k_2\}}} \rightarrow T\overline{F_{\{k_1\}^c}}\]
  for $k_1,k_2 \in \{1,\dots,\mu+1\}$ satisfying $k_1 \neq k_2$.
  Hence, the collection $\mathcal{E}_{a}$ generates $\overline{F_{\{k_1,k_2\}^c}}$ 
  for all $k_1 \neq k_2$ by the isomorphism 
  $\overline{F_{\{k_1,k_2\}^c}} \simeq T^{-1}\tau^{\mu-1}\overline{F_{\{k_1,k_2\}}}$.
  Repeating inductively this argument,
  we have that the collection $\mathcal{E}_{a}$ generates $\overline{F_{\{k_1,k_2,\dots,k_l\}^c}}$ 
  for all $l=1,...,\mu-1$ and $k_1 \neq k_2 \neq \cdots \neq k_{l}$. 

\begin{cl}\label{V}
  The collection $\mathcal{E}_{a}$ generates
  the objects $\tau^{l}\overline{F_{I}}$ for any subset $I$ and $l=0,1,\dots,\mu$.
\end{cl}

  By Claim~\ref{IV}, for any subset $I$, we have the exact triangle
  $$ \overline{F_{\{k_1,k_2,\dots,k_l\}}} \rightarrow \overline{F_{I\cup \{k_1,k_2,\dots,k_l\}}} 
  \rightarrow \tau^{l}\overline{F_{I}} \rightarrow T\overline{F_{\{k_1,k_2,\dots,k_l\}}}, $$
  where $k_1,\dots,k_l$ are some distinct elements of $I^c$ and $0< l \le \mu-|I|$.
  On the other hand, substituting $I^c$ for $I$, we have
  $\langle\mathcal{E}_{a}\rangle$ contains $\tau^{l'}\overline{F_{I^c}}$
  for $0< l' \le \mu-|I^c|=|I|-1$. 
  Moreover, we have $\langle\mathcal{E}_{a}\rangle$ contains $\overline{F_{I}}$ and 
  $T\overline{F_{I^c}}\simeq \tau^{|I^c|}\overline{F_{I}}$.
  By the isomorphism
  $T\tau^{l'}\overline{F_{I^c}} \simeq \tau^{l'+|I^c|}\overline{F_{I}}$,
  the collection $\mathcal{E}_{a}$ generates the objects $\tau^{l}\overline{F_{I}}$
  for all $l=0,\dots,\mu$.

Summarizing, we have obtained the objects $\tau^{l}\overline{F_{I}}$ 
for any subset $I$ and $l=0,1,\dots,\mu$.
By Proposition~\ref{T2}, this concludes the proof of lemma.
\end{proof}

\subsection{$\langle\mathcal{E}_{a}\rangle$ contains $(R/\mathfrak{m})^{stab}$}

Next, we will show that $\langle\mathcal{E}_{a}\rangle$ satisfies the condition (ii) 
in Theorem \ref{categorygenerating}.
By Lemma~\ref{closed}, the object $(R/\mathfrak{m})^{stab}$ is obtained by taking
the mapping cone $\overline{F_{\{i\}}}\rightarrow \tau\overline{F_{\{i\}}}$ 
and grading shift as follows.
The multiplication by $q$ on $R/\mathfrak{p_i}$ 
in $\gr\text{-}{R}$ induces the following exact triangle
\[(R/(x+s_iq))^{stab} \xrightarrow{(q,q)} \tau(R/(x+s_iq))^{stab} 
\rightarrow C((q,q)) \rightarrow T(R/(x+s_iq))^{stab}\]
in $\HMF^{gr}_{S}(\tilde{f_q})$. This is equivalent to
\[ \overline{F_{\{i\}}}\xrightarrow{(q,q)} \tau\overline{F_{\{i\}}} 
\rightarrow C((q,q)) \rightarrow T\overline{F_{\{i\}}}\]
in $\langle\mathcal{E}_{a}\rangle$. By the definition of mapping cone,
the matrix factorization $C((q,q))$ has the form
\[  \Bigl(\mf{\tau S \oplus \tau S}{f_0}{f_1} 
  {\tau^{h}S \oplus \tau^{2}S}\Bigr),\]
\[ f_0=\begin{pmatrix}
      -\prod_{j \neq i}(x+s_j q)  &  0 \\
      q  &  x +s_i q\\
      \end{pmatrix},
  f_1=\begin{pmatrix}
      -(x +s_i q)   &  0\\
      q             &  \prod_{j \neq i}(x+s_j q) \\
\end{pmatrix}.\]
By the direct caluculations, one obtains that $(R/\mathfrak{m})^{stab} \simeq T^{-1}\tau^{h-2}C((q,q))$
Since $\langle\mathcal{E}_{a}\rangle$ is closed under grading shift, the full subcategory  
$\langle\mathcal{E}_{a}\rangle$ contains $(R/\mathfrak{m})^{stab}$.

This completes the proof of Theorem~\ref{mainthm}.
\bibliographystyle{plain}
\bibliography{reference}
\end{document}